\documentclass[12pt]{amsart}
\usepackage{amsmath, amsthm, amscd, amsfonts}

\makeatletter \oddsidemargin.9375in \evensidemargin \oddsidemargin
\newtheorem{theorem}{Theorem}[section]

\newtheorem{corollary}[theorem]{Corollary}
\theoremstyle{definition}

\newtheorem{example}[theorem]{Example}

\theoremstyle{remark}
\newtheorem{remark}[theorem]{Remark}
\numberwithin{equation}{section}

\begin{document}

\noindent {\footnotesize

\title[The square integrable representation on $\mathbb{H}(G_\tau)$ ] 
{The square integrable representations on generalized Weyl- Heisenberg groups}

\author[F. Esmaeelzadeh]{Fatemeh Esmaeelzadeh}

\address{$^{1}$Department of Mathematics, Bojnourd Branch, Islamic Azad University, Bojnourd, Iran.
\newline
} \email{esmaeelzadeh@bojnourdiau.ac.ir}

}



\begin{abstract}
This paper presents the square integrable  representations of  generalized Weyl-Heisenberg group. We investigate the quasi regular   representation of generalized Weyl-Heisenberg group.  Moreover, we obtain a concrete form for admissible vector of this representation .  Finally, we provide  some examples to support our technical considerations.
\end{abstract} \maketitle
\section{Introduction }
Wavelet transform has rich theoretical  structures  and is extremely useful as tools for building signal transforms, adapted to various signal geometries, quantum mechanics, etc. Continuous wavelet transform admits a  generalization to locally compact groups. Such a unified approach seems to be useful, since it emphasizes on a clear way to basic features of continuous wavelet transform and includes all important cases  for applications \cite{olofson, Fuhr, wong}. It should be mentioned that, the Weyl Heisenberg group plays a significant  designations in various aspects of the connections between the classical harmonic analysis and concrete applications of numerical harmonic analysis. \\

This paper contains $4$ sections. Section 2  includes  the definition of semi-direct product  of two locally compact groups and generalized Weyl Heisenberg group. In Section 3, we study the square integrable representations on the generalized  Weyl Heisenberg group and then we obtain the necessary   and sufficient  conditions for admissible wavelet on this group. In Section 4, some examples are proved  as application of our results.
\section{Preliminaries and notation}
Let $H$ and $K$ be two locally compact groups with the identity elements $e_H$ and $e_K$, respectively and let $\tau:H\rightarrow Aut(K)$ be a homomorphism such that the  map $(h,k)\mapsto \tau_h(k)$ is continuous from $H\times K$ onto $K$, where $H\times K$ equips  with the product topology. The semi- direct product  topological group $G_\tau=H \times_\tau K$ is the locally compact topological space $H\times K$ under the product topology, with the group operations: $$(h_1,k_1)\times_\tau (h_2,k_2)=(h_1h_2,k_1\tau_{h_1}(k_2),$$
    $$(h,k)^{-1}=(h^{-1},\tau_{h^{-1}}(k^{-1})).$$ It is worth to note that  $K_1=\lbrace (e_H,k); k \in K\rbrace$  is a closed normal subgroup and $H_1=\lbrace (h,e_K); h \in H \rbrace$  is a closed subgroup of $G_\tau$ such that $G_\tau=HK$ . Moreover, the left Haar measure of the locally compact group $G_\tau$ is $$d\mu_{G_\tau}(h,k)=\delta_H(h)d\mu_H(h)d\mu_K(k),
 $$ \\
in which $d\mu_H,d\mu_K$ are the  left Haar measures on $H$ and $K$, respectively   and $\delta_H: H\rightarrow (0,\infty)$ is a positive continuous homomorphism that satisfies $$d\mu_K(k)=\delta_H(h)d\mu(\tau_h(k)),$$ for  $h \in H, k \in K$. Moreover, the modular function $\Delta_{G_\tau}$ is $$\Delta_{G_\tau}=\delta_H(h)\Delta_H(h)\Delta_K(k),$$ where $\Delta_H, \Delta_K$  are the modular functions of  $H, K$, respectively. \\
 When $K$ is also abelian,  one can define  $\hat{\tau}:H \rightarrow Aut(\hat{K})$ via $h\mapsto \hat{\tau_h}$ where $$\hat{\tau_h}(\omega)=\omega \circ \tau_{h^{-1}},$$ for all $\omega \in{\hat{K}}$.
 We usually denote $\omega \circ \tau_{h^{-1}}$ by $\omega_h$. With this notation, it is easy to see $$\omega_{h_1h_2}=(\omega_{h_2})_{h_1},$$ where 
 $h_1,h_2 \in H$ and $\omega \in{\hat{K}}$.
   The semi-direct product $G_{\hat{\tau}}=H \times_ {\hat{\tau}} \hat{K}$ is a locally compact group with the left Haar measure 
  $$d\mu_{\hat{G}}(h,\omega)=\delta_H(h)^{-1} d\mu_H(h)d\mu_{\hat{K}}(\omega),$$
  where $d\mu_{\hat{K}}$ is the Haar measure on $\hat{K}$. Also, for all $h \in H$, $$d\mu_{\hat{K}}(\omega_h)=\delta_H(h)d\mu_{\hat{K}}(\omega),$$  for $\omega \in{\hat{K}},$ (see more details  in \cite{ghaani, Arefi, Fuhr}.)\\
Let $G_\tau=H \times_\tau K$, and define $\theta:G_\tau \rightarrow Aut(\hat{K}\times \mathbb{T})$ via $$(h,k)\mapsto \theta_{(h,k)}(\omega,z)=(\hat{\tau_h}(\omega),  \hat{\tau_h}(\omega)(k)z)=(\omega_h,\omega_h(k)z),$$ for all $(h,k) \in{H \times_\tau K}$
 and $(\omega,z)\in{\hat{K}\times \mathbb{T}}.$ The mapping $\theta$ is a continuous homomorphism. Thus the semi-direct prodoct $$G_\tau \times_\theta (\hat{K}\times \mathbb{T})=(H \times_\tau K)\times_\theta (\hat{K}\times \mathbb{T}),$$ is a locally compact group and it is called the generalized Weyl Heisenberg group associated with the semi direct product group $G_\tau =H \times_\tau K$, and denoted by $\mathbb{H}(G_\tau)$. It is easy to see that  the
  group operations of $\mathbb{H}(G_\tau)$ are
   $$(h_1,k_1,\omega_1,z_1).(h_2,k_2,\omega_2,z_2)=(h_1h_2,k_1\tau_{h_1}(k_2),\omega_1{\omega_2}_{h_1}, {\omega_{2}}_{h_1}(k)z_1z_2),$$ $$(h_1,k_1,\omega_1,z_1)^{-1}=(h_1^{-1},\tau_{h_1}^{-1}(k^{-1}),\bar{\omega}_{h_1^{-1}}, \bar{\omega}_{h_1^{-1}}(\tau_{h_1}^{-1}(k^{-1}))z^{-1}),$$
    for $(h_1,k_1,\omega_1,z_1),(h_2,k_2,\omega_2,z_2) \in{\mathbb{H}(G_\tau)}$ (see \cite{ghaani})  and the   left Haar measure of $\mathbb{H}(G_\tau)$  is: $$d\mu_{\mathbb{H}(G_\tau)}(h,k,\omega,z)=d\mu_H(h)d\mu_K(k)d\mu_{\hat{K}}(\omega)d\mu_{\mathbb{T}}(z).$$

\section{The square integrable representation of $\mathbb{H}(G_\tau)$}
Throughout this section, we assume that $H$ and $K$ are locally compact topological groups and that $K$ is abelian, too. We denote the left Haar measures of $H$ and $K$ by $d\mu_H, d\mu_K$, respectively. Suppose that $h\mapsto \tau_h$ from $H$ to $Aut(K)$ is a homomorphism such that $(h,k)\mapsto \tau_h(k)$ from $H \times K$ into $K$ is continuous.  $G_\tau=H \times_\tau K$ is the semi-direct product of $H$ and $K$  that is a locally compact topology group with the left Haar measure $d\mu_{G_\tau}(h,k)=\delta_H(h) d\mu_H((h)d\mu_K(k)$, where $\delta_H: H \mapsto (0,\infty)$ is a  continuous homomorphism. Consider the homomorphism $\theta: G_\tau \rightarrow Aut(\hat{K}\times \mathbb{T})$ is defined by  $$((h,k),(\omega,z))\mapsto \theta_{(h,k)} (\omega,z),$$ where $\theta_{(h,k)}(\omega,z)=(\omega \circ \tau_{h^{-1}}, \omega \circ \tau_{h^{-1}}(k).z).$ This makes $\mathbb{H}(G_\tau)=G_\tau \times_\theta( \hat{K}\times \mathbb{T})$  a locally compact topological group  where $\mathbb{H}(G_\tau)$ is equipped with the product topology and the group operations as $$(h_1,k_1,\omega_1,z_1).(h_2,k_2,\omega_2,z_2)=(h_1h_2,k_1\tau_{h_1}(k_2),\omega_1{\omega_2}_{h_1}, {\omega_2}_{h_1}(k)z_1z_2),$$ $$(h_1,k_1,\omega_1,z_1)^{-1}=(h_1^{-1},\tau_{h_1}^{-1}(k^{-1}),\bar{\omega}_{h_1^{-1}}, \bar{\omega}_{h_1^{-1}}(\tau_{h_1}^{-1}(k^{-1}))z^{-1}),$$
for $(h_1,k_1,\omega_1,z_1),(h_2,k_2,\omega_2,z_2) \in{\mathbb{H}(G_\tau)}.$ The   left Haar measure of $\mathbb{H}(G_\tau)$ is $$d\mu_{\mathbb{H}(G_\tau)}(h,k,\omega,z)=d\mu_H(h)d\mu_K(k)d\mu_{\hat{K}}(\omega)d\mu_{\mathbb{T}}(z).$$\\

Now, we are going to define a square integrable representation on $\mathbb{H}(G_\tau)$. With the above notations define $\pi:\mathbb{H}(G_\tau) \rightarrow U(L^2(\hat{K}))$ by
\begin{equation}
\pi(h,k,\omega,z)f(\xi)=\delta_H^{-1/2}(h)z\xi(k) \overline{\omega(k)}f((\xi\overline{\omega})_{h^{-1}}),
\end{equation}
then $\pi$ is a homomorphism. Indeed,
\\$\begin{array}{lll}
\pi\left( (h_1,k_1,\omega_1,z_1)(h_2,k_2,\omega_2,z_2)\left) f(\xi)=\pi(h_1h_2,k_1\tau_{h_1}(k_2),
\omega_{1}\right( \omega_{2}\right) _{h_{1}},(\omega_{2})_{h_{1}}(k_{1})z_{1}z_{2})f(\xi)\\[1ex]=
\delta_H^{-1/2}(h_{1}h_{2})
(\omega_2)_{h_1}(k_1)z_1z_2\xi(k_1\tau_{h_1}(k_2)) \overline{\omega_1(\omega_2)_{h_1}}(k_1\tau_{h_1}
(k_2)) f((\xi \overline{\omega_1(\omega_2)_{h_1}})_{(h_1h_2)^{-1}}
\\[1ex]=\delta_H^{-1/2}(h_1h_2)(\omega_2)_{h_1}(k_1)z_1z_2\xi(k_1)\xi_{h_1^{-1}}(k_2)\overline{\omega_1(k_1)}
\overline{(\omega_1)_{h_1^{-1}}(k_2)}\overline{\omega_2(k_2)}f(\xi_{h_2^{-1}h_1^{-1}}
\overline{(\omega_1)}_{h_2^{-1}h_1^{-1}}
\overline{(\omega_2)}_{h_2^{-1}})
\end{array}.$\\
Also,
\\$\begin{array}{lll}
\pi(h_1,k_1,\omega_1,z_1)\pi(h_2,k_2,\omega_2,z_2)f(\xi)\\[1ex] =\delta_H^{-1/2}(h_1)z_1\xi(k_1)\overline{\omega_1}(k_1)\pi(h_2,k_2,\omega_2,z_2)f((\xi \overline{\omega_1})_{h_1^{-1}}\\[1ex]=\delta_H^{-1/2}(h_1)\delta_H^{-1/2}(h_2)z_1z_2\xi(k_1)\overline{\omega_1}(k_1)\overline{\omega_2}(k_2)(\xi \overline{\omega_1})_{h_1^{-1}}(k_2)f((\xi \overline{\omega_1})_{h_1^{-1}}(\overline{\omega_2})_{h_2^{-1}})
\\[1ex]=\delta_H^{-1/2}(h_1h_2)z_1z_2 \xi(k_1)\xi_{h_1^{-1}}(k_2)\overline{\omega_1(k_1)}
\overline{(\omega_1)_{h_1^{-1}}(k_2)}\overline{\omega_2(k_2)}f(\xi_{h_2^{-1}h_1^{-1}}
\overline{(\omega_1)}_{h_2^{-1}h_1^{-1}}
\overline{(\omega_2)}_{h_2^{-1}}).
\end{array}$.\\

Moreover, $\pi$ is unitary. In fact we have,
\\$\begin{array}{rcl}
\Vert \pi(h,k,\omega,z)f\Vert^2_2 &=&\int_{\hat{K}}\vert \pi(h,k,\omega,z)f(\xi)\vert^2 d\mu_{\hat{K}}(\xi)\\[1ex]&=& \int_{\hat{K}}\delta_H^{-1}(h)\vert f((\xi \overline{\omega})_{h^{-1}}\vert^2d\mu_{\hat{K}}(\xi)\\[1ex]&=&\int_{\hat{K}}\delta_H^{-1}(h)\vert f((\xi )_{h^{-1}}\vert^2d\mu_{\hat{K}}(\xi)\\[1ex]&=&\int_{\hat{K}}f((\xi )\vert^2d\mu_{\hat{K}}(\xi)\\[1ex]&=&\Vert f \Vert_2^2.
\end{array}$.\\
And it is easy to check that $\pi$ is continuous and onto. So, $\pi$ is a continuous unitary representation of group $\mathbb{H}(G_\tau)$ to the Hilbert space $L^2(\hat{K})$. In the sequel, we show that $\pi$ is irreducible  when $H$ is compact. Furthermore, it is also shown that $\pi$ is square integrable if and only if $H$ is compact. Note that when $H$ is a compact group, we normalize the Haar measure $\mu_H$ such that $\mu_H(H)=1$. 
\begin{theorem}\label{theo}
Let $\mathbb{H}(G_\tau)=(H \times_\tau K)\times_\theta(\hat{K}\times \mathbb{T})$ where $H$ is  a locally  compact group and $K$ is a locally compact  abelian group. Then for $\varphi, \psi $ in $L^2(\hat{K})$,
\begin{equation}\label{equ}
\int_{\mathbb{H}(G_\tau)} \vert \prec \varphi, \pi(h,k,\omega,z)\psi\succ\vert^2 d\mu_{\mathbb{H}(G_\tau)}(h,k,\omega,z)=\Vert \varphi \Vert _2^2 \Vert \psi \Vert_2^2.
\end{equation} 
if and only if $H$ is compact.
\begin{proof}
For $\varphi, \psi $ in $L^2(\hat{K})$ we first consider the following observations:
\\$\begin{array}{lll}
\int_{\mathbb{H}(G_\tau)} \vert \prec \varphi, \pi(h,k,\omega,z)\psi\succ\vert^2 d\mu_{\mathbb{H}(G_\tau)}(h,k,\omega,z)\\[1ex]=\int_{\mathbb{H}(G_\tau)}  \vert \int_{\hat{K}} \varphi(\xi) \overline{\pi(h,k,\omega,z)\psi(\xi)}d\mu_{\hat{K}}(\xi)\vert^2 d\mu_{\mathbb{H}(G_\tau)}(h,k,\omega,z)\\[1ex]=\int_{\mathbb{H}(G_\tau)}  \vert \int_{\hat{K}} \varphi(\xi) \delta_H^{-1/2}(h) \overline{z} \overline{\xi(k)}\omega(k) \overline{\psi(\xi \overline{\omega})_{h^{-1}}}d\mu_{\hat{K}}(\xi)\vert^2 d\mu_{\mathbb{H}(G_\tau)}(h,k,\omega,z)\\[1ex]=
\int_{\mathbb{H}(G_\tau)}  \vert \int_{\hat{K}} \varphi(\xi \omega) \delta_H^{-1/2}(h) \overline{z} \overline{\xi(k)}\overline{\psi(\xi )_{h^{-1}}}d\mu_{\hat{K}}(\xi)\vert^2 d\mu_{\mathbb{H}(G_\tau)}(h,k,\omega,z)\\[1ex]=
\int_{\mathbb{H}(G_\tau)}  \vert \int_{\hat{K}} R_{\omega}\varphi(\xi ) \delta_H^{-1/2}(h) \overline{z} \overline{\xi(k)}\overline{\psi(\xi \circ \tau_h)}d\mu_{\hat{K}}(\xi)\vert^2 d\mu_{\mathbb{H}(G_\tau)}(h,k,\omega,z)\\[1ex]=\int_{\mathbb{H}(G_\tau)}  \vert \int_{\hat{K}} R_{\omega}\varphi(\xi\circ \tau_{h^{-1}} ) \delta_H^{-1/2}(h) \overline{z} \overline{\xi \circ \tau_{h^{-1}}  (k)}\overline{\psi(\xi) }d\mu_{\hat{K}}(\xi_h)\vert^2 d\mu_{\mathbb{H}(G_\tau)}(h,k,\omega,z)\\[1ex]=
\int_{\mathbb{H}(G_\tau)}  \vert \int_{\hat{K}} R_{\omega}\varphi(\xi\circ \tau_{h^{-1}} ) \delta_H^{1/2}(h) \overline{z} \overline{\xi (\tau_{h^{-1}}  (k))}\overline{\psi(\xi) }d\mu_{\hat{K}}(\xi)\vert^2 d\mu_{\mathbb{H}(G_\tau)}(h,k,\omega,z)\\[1ex]=\int_{\mathbb{H}(G_\tau)}  \delta_H(h) \vert \int_{\hat{K}} (R_{\omega}\varphi(.\circ \tau_{h^{-1}}) .\overline{\psi})(\xi)  \overline{\xi (\tau_{h^{-1}}  (k))}d\mu_{\hat{K}}(\xi)\vert^2 d\mu_{\mathbb{H}(G_\tau)}(h,k,\omega,z)\\[1ex]=
\int_{\mathbb{H}(G_\tau)}  \delta_H(h) \widehat{\vert (R_{\omega}\varphi(.\circ \tau_{h^{-1}} ).\overline{\psi})}   (\tau_{h^{-1}}  (k))\vert^2 d\mu_{\mathbb{H}(G_\tau)}(h,k,\omega,z)\\[1ex]=
\int_H \delta_H(h)  \int_{\hat{K}}\int_K \vert \widehat{(R_{\omega}\varphi(.\circ \tau_{h^{-1}} ).\overline{\psi})}   (\tau_{h^{-1}}  (k))\vert^2 d\mu_K(k)d\mu_{\hat{K}}(\omega)d\mu_H(h)\\[1ex]=
\int_H   \int_{\hat{K}}\int_K \vert \widehat{(R_{\omega}\varphi(.\circ \tau_{h^{-1}}) .\overline{\psi})}   (k)\vert^2 d\mu_K(k)d\mu_{\hat{K}}(\omega)d\mu_H(h)\\[1ex]=
\int_H   \int_{\hat{K}}\int_{\hat{K}} \vert (R_{\omega}\varphi(.\circ \tau_{h^{-1}}) .\overline{\psi})   (\xi)\vert^2 d\mu_{\hat{K}}(\xi)d\mu_{\hat{K}}(\omega)d\mu_H(h)\\[1ex]=
\int_H   \int_{\hat{K}}\int_{\hat{K}} \vert R_{\omega}\varphi(\xi\circ \tau_{h^{-1}}) .\overline{\psi}(\xi)   \vert^2 d\mu_{\hat{K}}(\xi)d\mu_{\hat{K}}(\omega)d\mu_H(h)\\[1ex]=\int_H   \int_{\hat{K}}\int_{\hat{K}} \delta_H(h)\vert R_{\omega}\varphi(\xi) .\overline{\psi}(\xi \circ \tau_h)   \vert^2 d\mu_{\hat{K}}(\xi)d\mu_{\hat{K}}(\omega)d\mu_H(h)\\[1ex]=\int_H   \int_{\hat{K}}\Vert \varphi \Vert_2^2\delta_H(h)\vert \overline{\psi}(\xi \circ \tau_h)   \vert^2 d\mu_{\hat{K}}(\xi)d\mu_H(h)\\[1ex]=
\Vert \varphi \Vert_2^2 \Vert \psi \Vert_2^2 \mu_H(H)
\end{array}$.\\
 Now, if $H$ is compact, then $\mu_H(H)=1$. So, (\ref{equ})   holds. Conversely, if (\ref{equ})  holds,the above observation implies that $\mu_H(H)=1$  . So, we can conclude that $H$ is compact.
\end{proof}
\end{theorem}
\begin{corollary}
With notation as above, the representation $\pi$ of $\mathbb{H}(G_\tau)$ on $L^2(\hat{K})$  is irreducible  if $H$ is compact. 
\begin{proof}
If $H$ is compact, then 
(\ref{equ}) in Theorem 3.1 holds. Now, 
suppose that $M$ is a closed subspace of the Hilbert space  $L^2(\hat{K})$  that is invariant under $ \pi$. Then for any $\varphi \in M$ we have,
 $$\lbrace \pi (h,k, \omega ,z) \varphi; \ \  (h,k, \omega ,z) \in{\mathbb{H}(G_\tau)} \rbrace \subseteq M.$$ Let $\psi \in{L^2(\hat{K})}$ be  orthogonal to $M$, that is $\prec \psi, \pi(h,k, \omega ,z) \varphi \succ=0,$ for all  $(h,k, \omega ,z) \in{\mathbb{H}(G_\tau)}$.  Thus by (\ref{equ}), $\Vert \varphi \Vert_2\Vert \psi \Vert_2=0$, and hence $\psi=0$. So, $M^{\perp}=\lbrace 0 \rbrace$, that is, $M=L^2(\hat{K})$. Namely, $\pi$ is irreducible. 
\end{proof}
\end{corollary}
We remind the reader that, an irreducible representation $\pi$  of $\mathbb{H}(G_\tau)$ on $L^2(\hat{K})$ is called square integrable if there exists  a non zero element $\psi$ in $L^2(\hat{K})$ such that 
\begin{equation}\label{form}
\prec \pi(., ., ., .)\psi,f  \succ \in{L^2({\mathbb{H}(G_\tau)})},
\end{equation}
 for all $f \in{L^2(\hat{K})}$. A unit vector $\psi$ satisfying (\ref{form}) is said to be an admissible wavelet for $\pi$, and the constant $$c_{\psi}=\int_{\mathbb{H}(G_\tau)}\vert \prec \pi(h,k,\omega,z)\psi, \psi\succ \vert^2d\mu_{\mathbb{H}(G_\tau)},$$ is called the wavelet constant associated to the admissible wavelet $\psi$.\\
Also, for the wavelet vector $\psi$, the continuous wavelet transform is defined by $$W_\psi f(h,k,\omega,z)=\prec f, \pi(h,k,\omega,z)\psi\succ.$$ It is easy to see that $(h,k,\omega,z) \mapsto W_\psi f(h,k,\omega,z)$ is a continuous function  on $\mathbb{H}(G_\tau).$ Moreover, $W_\psi$ intertwines $\pi$ and  the left regular representation  on $\mathbb{H}(G_\tau).$ 
\begin{corollary}
The   representation $\pi$ of the  $GWH$ group $\mathbb{H}(G_\tau)=(H\times_\tau K)\times_\theta(\hat{K}\times \mathbb{T})$ on $L^2(\hat{K})$ is square integrable if and only if $H$ is  compact.
\end{corollary}
\begin{proof}
If $H$ is compact, then by Theorem 3.1 and Corollary 3.2, $\pi$ is square integrable. For the inverse, if $\pi$ is square integrable, then there exists a non zero element $\varphi \in L^2(\hat{K})$ such that 
\begin{equation*}
\prec \pi(., ., ., .)\varphi,\psi  \succ \in{L^2({\mathbb{H}(G_\tau)})},
\end{equation*}
 for all $\psi \in{L^2(\hat{K})}$. On the other hand, 
\begin{equation*}
\int_{\mathbb{H}(G_\tau)} \vert \prec \varphi, \pi(h,k,\omega,z)\psi\succ\vert^2 d\mu_{\mathbb{H}(G_\tau)}(h,k,\omega,z)=\Vert \varphi \Vert _2^2 \Vert \psi \Vert_2^2 \mu_H(H).
\end{equation*} 
So $\mu_H(H)< \infty$. That is $H$ is compact.
\end{proof}
\begin{remark}
There is another irreducible  representation of $\mathbb{H}(G_\tau)$ on Hilbert space $L^2(K)$. Indeed, consider
\begin{equation*}\label{rep}
\tilde{\pi}:\mathbb{H}(G_\tau) \rightarrow U(L^2(K)), \ \ \ \ \tilde{\pi}(h,k,\omega,z)f(k')=\delta_H(h)^{1/2}z \omega(k')f(\tau_{h^{-1}}(k'k)),
\end{equation*}
for all $(h,k,\omega,z) \in{\mathbb{H}(G_\tau)}, f \in{L^2(K)}$.  $\tilde{\pi}$ is homomorphism and unitary. In fact we have
\\$\begin{array}{lll}
 \tilde{\pi}(( h_1,k_1,\omega_1,z_1)(h_2,k_2,\omega_2,z_2 ))  f(k')\\[1ex]=\tilde{\pi}(h_1h_2,k_1\tau_{h_1}(k_2),\omega_1(\omega_2)_{h_1},(\omega_2)_{h_1}(k_1)z_1z_2)f(k')\\[1ex]=\delta_H^{1/2}(h_1h_2)(\omega_2)_{h_1}(k_1)z_1z_2 \omega_1(\omega_2)_{h_1}(k') f(\tau_{(h_1h_2)^{-1}}(k'(k_1\tau_h{}_1(k_2))
\\[1ex]=\delta_H^{1/2}(h_1h_2)z_1z_2\omega_1(k')(\omega_2)_{h_1}(k'k_1)f(\tau_{h_2^{-1}h_1^{-1}}(k'k_1\tau_{h_1}(k_2))\\[1ex]=\delta_H^{1/2}(h_1h_2)z_1z_2\omega_1(k')\omega_2)_{h_1}(k'k_1)f(\tau_{h_2^{-1}h_1^{-1}}(k'k_1)\tau_{h_2^{-1}}(k_2)),
\end{array}$\\
and
\\$\begin{array}{lll}
\tilde{\pi}(h_1,k_1,\omega_1,z_1)\tilde{\pi}(h_2,k_2,\omega_2,z_2)f(k')\\[1ex] =\delta_H^{1/2}(h_1)z_1\omega_1(k')\tilde{\pi}(h_2,k_2,\omega_2,z_2)f(\tau_{h_1^{-1}}(k'k_1)) \\[1ex]=\delta_H^{1/2}(h_1)\delta_H^{1/2}(h_2)z_1z_2\omega_1(k')\omega_2(\tau_{h_1^{-1}}(k'k_1))f(\tau_{h_2^{-1}}(\tau_{h_1^{-1}}(k'k_1)k_2)
\\[1ex]=\delta_H^{1/2}(h_1h_2)z_1z_2 \omega_1(k')(\omega_2)_{h_1}(k'k_1)f(\tau_{h_2^{-1}h_1^{-1}}(k'k_1)\tau_{h_2^{-1}}(k_2)).
\end{array}$.\\
Also,
\\$\begin{array}{rcl}
\Vert \tilde{\pi}(h,k,\omega,z)f\Vert^2_2 &=&\int_{K}\vert \tilde{\pi}(h,k,\omega,z)f(k')\vert^2 d\mu_{K}(k')\\[1ex]&=& \int_{K}\delta_H(h)\vert f(\tau_{h^{-1}}(k'k))\vert^2d\mu_{K}(k')\\[1ex]&=&\int_{K}\delta_H(h)\vert f(k')\vert^2d\mu_{K}(\tau_h(k'))\\[1ex]&=&\int_{K}\vert f(k')\vert^2d\mu_{K}(k')\\[1ex]&=&\Vert f \Vert_2^2.
\end{array}$.\\
Using the Plancherel theorem, $\pi, \tilde{\pi}$ are unitarily equivalent. So, $\tilde{\pi}$ is square integrable if and only if $\pi$ is square integrable.
\end{remark}

\begin{remark}
The inverse of Corollary 3.2 does not hold, generally. An obvious example is when $H$ is a non compact group and $K$ is the  trivial group $\lbrace e \rbrace$. Then the representation $\pi:\mathbb{H}(H \times_\tau \lbrace e \rbrace )\rightarrow U(\mathbb{C})$ is an irreducible representation. Here we give a non trivial example in which $\pi$ is  an irreducible representation, but $H$ is not compact. Let $H=\mathbb{R}^+, K=\mathbb{R}$. Define the representation $\pi$ of $\mathbb{H}(\mathbb{R}^+\times_\tau \mathbb{R})$ as follows:
\begin{equation*}
\pi:\mathbb{H}(\mathbb{R}^+\times_\tau \mathbb{R})\rightarrow U(L^2(\mathbb{R})); \ \ \ \ \pi(a,x,\omega,z)f(\xi)=a^{1/2} z e^{2\pi ix(\xi-\omega)}f((\xi \bar{\omega)}_{a^{-1}}),
\end{equation*}
in which $(\xi \bar{\omega)}_{a^{-1}}=(\xi \bar{\omega)} \circ \tau_a$,  $\tau_a(x)=a.x$ and $\delta_H(a)=a^{-1}$. This representation is irreducible. Indeed, let $M$ be a closed invariant subspace of $L^2(\mathbb{R})$ under $\pi$. Then for any $f\in{M}$, we have $\pi(h,k,\omega,z)f \in M$. Consider $0 \neq g\in{M^{\perp}}$, so that  $\prec g, \pi(h,k,\omega,z)f\succ =0$. Then 
\begin{equation*}
0=\int_{\mathbb{R}}g(\xi)e^{-2\pi i x \xi}\bar{f}((\xi \bar{\omega})_{a^{-1}})d\xi=\\[1ex]\int_{\mathbb{R}}g(\xi_a \omega)e^{-2\pi i x \xi_a \omega}\bar{f}(\xi )d\xi.
\end{equation*}
 Thus, $g(\xi_a \omega)\bar{f}(\xi )=0$, for almost all $\xi \in \mathbb{R}$. Suppose that $\bar{f}(\xi ) \neq 0$, for all $\xi$ in a set $A$ with positive measure. Then for all $\xi \in A, \ \  g(\xi_a \omega)=0$, for all $\omega \in{\mathbb{R}}, a \in{\mathbb{R}^+}$. Thus $g=0$. This is a contradiction. So, $\pi$ is an irreducible representation, but $H$ is not compact.
\end{remark}

In the sequel, we define the quasi regular representation and we obtain a concrete form for an  admissible vector.
Note that $\mathbb{H}(G_\tau)$ acts on the Hilbert space 
$L^2(\hat{K}\times \mathbb{T})$  and this action induces the quasi regular representation $\lbrace \rho, L^2(\hat{K}\times \mathbb{T})\rbrace$ as follows:
\begin{equation}\label{pi}
\rho:(H \times_\tau K) \times_\theta( \hat{K}\times \mathbb{T}) \rightarrow U(L^2(\hat{K}\times \mathbb{T})),
\end{equation}

where
\begin{eqnarray*}
\rho(h,k,\omega,z)f(\xi,t)&=&\delta_{H \times_\tau K}^{1/2}(h,k)f(\theta_{(h,k)^{-1}}(\xi,t)(\omega,z)^{-1})\\[1ex]&=&\delta_{H}^{-1/2}(h)f(\theta_{(h^{-1},\tau_{h^{-1}}({k^{-1}})}(\xi \bar{\omega},tz^{-1}))\\[1ex]&=&\delta_{H}^{-1/2}(h) f((\xi \bar{\omega})_{h^{-1}},(\xi \bar{\omega})_{h^{-1}}(\tau_{h^{-1}}(k^{-1})).tz^{-1}).
\end{eqnarray*}
Note that $\delta_{H \times_\tau K}(h,k)=\delta_H(h)^{-1}.$ (see Corollary 3.3 in \cite{ghaani})\\
 A type of the Fourier transform of the quasi regular representation $\rho$ obtains as  follows:
\\$\begin{array}{lll}
\widehat{\rho(h,k,\omega,z)f}(k',n')\\[1ex]= \int_{\hat{K}\times \mathbb{T}}\rho(h,k,\omega,z)f(\xi,t) \overline{(k',n')(\xi,t)}d\mu_{\hat{K}}(\xi)d\mu_{\mathbb{T}}(t)\\[1ex]= \delta_H(h)^{-1/2}\int_{\hat{K}\times \mathbb{T}} f((\xi \overline{\omega})_{h^{-1}},(\xi \overline{\omega})_{h^{-1}}(\tau_{h^{-1}}(k^{-1}))tz^{-1}) \overline{\xi(k')} \overline{ t^{n'}} d\mu_{\hat{K}}(\xi)d\mu_{\mathbb{T}}(t)\\[1ex]= \delta(h)^{-1/2} \overline{z^{n'}}\int_{\hat{K}\times \mathbb{T}}f(\xi )_{h^{-1}},\xi _{h^{-1}}(\tau_{h^{-1}}(k^{-1}))t) \overline{\omega}(k')\bar{\xi(k')} \overline{ t^{n'}} d\mu_{\hat{K}}(\xi)d\mu_{\mathbb{T}}(t)\\[1ex]=\delta_H(h)^{-1/2} \bar{z^{n'}}\bar{\omega}(k')\int_{\hat{K}\times \mathbb{T}}f(\theta_{(h^{-1},\tau_{h^{-1}})}(\xi,t)\bar{\xi(k')} \bar{ t^{n'}} d\mu_{\hat{K}}(\xi)d\mu_{\mathbb{T}}(t)\\[1ex]= \delta_H(h)^{-1/2} \bar{z^{n'}}\bar{\omega}(k')\int_{\hat{K}\times \mathbb{T}}f \circ \theta_{(h,k)^{-1}}(\xi,t)  \overline{(k',n')(\xi,t)}d\mu_{\hat{K}}(\xi)d\mu_{\mathbb{T}}(t)\\[1ex]=  \delta_H(h)^{-1/2} \bar{z^{n'}}\bar{\omega}(k')\widehat{(f \circ \theta_{(h,k)^{-1}})}(k',n'),
\end{array}$\\\\
 for all $(k',n')\in{K \times \mathbb{Z}}=\widehat{(\hat{K}\times \mathbb{T})}.$\\ So,
\begin{equation}\label{1}
\widehat{\rho(h,k,\omega,z)f}(k',n')=\delta_H(h)^{-1/2} \bar{z^{n'}}\bar{\omega}(k')\widehat{(f \circ \theta_{(h,k)^{-1}})}(k',n').
\end{equation}

\begin{theorem}
With the notation as above, let $\rho$ be the quasi regular representation on $\mathbb{H}(G_\tau),$ and $\psi, f \in{L^2(\hat{K}\times \mathbb{T})}$.\\
\item[(i)]  If $\psi$ is a wavelet vector, then 
$$W_\psi f(h,k,\omega,z)= \delta_H^{-1/2}(h) \int_K \sum_{n' \in{\mathbb{Z}}} \hat{f}(k',n')z^{n'}\omega(k')\overline{\widehat{{(\psi \circ \theta)}}}_{(h,k)^{-1}}(k',n')d\mu_K(k').$$
\item[(ii)] The vector $\psi$ is wavelet if 
$$\int_{H\times_\tau K} \vert \hat{\psi}(k',n') \circ \theta_{(h,k)^{-1}}\vert^2 d\mu_{H\times K}(h,k) < \infty.$$
\begin{proof}
For $(k',n') \in{K \times \mathbb{Z}},$ 
\item[(i)] By the Plancherel's theorem and (\ref{1}), we have 
\begin{eqnarray*}
W_\psi f(h,k,\omega,z)&=& \prec f, \rho(h,k,\omega,z)\psi \succ \\[1ex]&=& \prec  \hat{f}, \widehat{\rho(h,k,\omega,z)\psi}\succ \\[1ex]&=&  \delta_H^{-1/2}(h)\int_K \sum_{n' \in{\mathbb{Z}}} \hat{f}(k',n')z^{n'}\omega(k')\overline{\widehat{{(\psi \circ \theta)}}}_{(h,k)^{-1}}(k',n')d\mu_K(k').
\end{eqnarray*}
\item[(ii)] By applying  the part (i),  for $f \in{L^2(\hat{K}\times \mathbb{T})}$, we get
\\$\begin{array}{lll}
\int_{\hat{K}\times \mathbb{T}} \vert W_\psi f(h,k,\omega,z) \vert^2 d\mu_{\hat{K}\times \mathbb{T}}(\omega,z)\\[1ex]= \int_{\hat{K}\times \mathbb{T}}W_\psi f(h,k,\omega,z) \overline{W_\psi f(h,k,\omega,z)}d\mu_{\hat{K}\times \mathbb{T}}(\omega,z)\\[1ex]=   \delta_H^{-1}(h) \int_{\hat{K}\times \mathbb{T}} [(\int_K \sum_{n' \in{\mathbb{Z}}} \hat{f}(k',n')z^{n'}\omega(k')\overline{\widehat{{(\psi \circ \theta)}}}_{(h,k)^{-1}}(k',n')d\mu_K(k'))\\\times(\overline{\int_K \sum_{n'' \in{\mathbb{Z}}} \hat{f}(k',n')z^{n''}\omega(k'')\overline{\widehat{{(\psi \circ \theta)}}}_{(h,k)^{-1}}(k'',n'')}d\mu_K(k''))]\\[1ex] =\delta_H^{-1}(h) \int_{\hat{K}\times \mathbb{T}}\vert \hat{F}(\omega,z)\vert^2d\mu_{\hat{K}\times \mathbb{T}}\\[1ex] =\delta_H^{-1}(h) \int_{\hat{K}\times \mathbb{T}}\vert F(k',n')\vert^2d\mu_{{K}\times \mathbb{Z}}\\[1ex] =\delta_H^{-1}(h)\int_K \sum_{n' \in{\mathbb{Z}}} \vert \hat{f}(k',n')\vert^2 \vert\widehat{{(\psi \circ \theta)}}(k',n')\vert^2d\mu_K(k'), 
\end{array}$\\\\

where $\hat{F}=\hat{f}\hat{(\psi \circ \theta)}\in{L^1(K\times \mathbb{Z}}).$ It is easy to see that $$\widehat{(\psi \circ \theta)}((k',n')=\delta_H^{-1}(h) \hat{\psi}(k',n')\circ \theta_{(h,k)^{-1}}.$$
Then 
\begin{equation}\label{2.2}
\int_{\hat{K}\times \mathbb{T}} \vert W_\psi f(h,k,\omega,z) \vert^2 d\mu_{\hat{K}\times \mathbb{T}}(\omega,z)
= \delta_H^{-1}(h)\int_K \sum_{n' \in{\mathbb{Z}}} \vert \hat{f}(k',n')\vert^2 \vert\widehat{(\psi(k',n')} \circ \theta_{(h,k)^{-1}}\vert^2d\mu_K(k').
\end{equation} 
Now, by using (\ref{2.2}) we have 
\\$\begin{array}{lll}
\Vert W_\psi f\Vert_2^2 &=&\int_{\mathbb{H}(G_\tau)} \vert W_\psi f(h,k,\omega,z) \vert^2 d\mu_{\mathbb{H}(G_\tau)}(h,k,\omega,z)\\[1ex]&=&\int_{H\times_\tau K} \int_{\hat{K}\times \mathbb{T}}\vert W_\psi f(h,k,\omega,z) \vert^2 \delta_H^{-1}(h) d\mu_{\hat{K}\times \mathbb{T}}(\omega,z)d\mu_{H\times_\tau K} (h,k)\\[1ex]&=& \int_{H\times_\tau K} \int_K \sum_{n' \in{\mathbb{Z}}} \vert \hat{f}(k',n')\vert^2 \vert\widehat{(\psi(k',n')} \circ \theta_{(h,k)^{-1}}\vert^2d\mu_K(k')d\mu_{H\times_\tau K} (h,k)\\[1ex] &=&\Vert f \Vert_2^2 \int_{H\times_\tau K}\vert\widehat{(\psi(k',n')} \circ \theta_{(h,k)^{-1}}\vert^2d\mu_{H\times_\tau K} (h,k),
\end{array}$\\
and then the proof of part $(ii)$ is complete.
\end{proof}
\end{theorem}

\section{Examples and applications}
\begin{example}
Let $K$ be an abelian locally compact group and $H=\lbrace e \rbrace$ (the trivial group). In this case the generalized weyl Heisenberg group $\mathbb{H}(G_\tau)$ coincides with the standard weyl Heisenberg group $G:=K\times_\theta(\hat{K}\times \mathbb{T})$. In this case the square integrable representation of $G=K\times_\theta(\hat{K}\times \mathbb{T})$ on $L^2(\hat{K})$ is as follows:
\begin{equation}
\pi(k,\omega,z)f(\xi)=z\xi(k) \overline{\omega(k)} f(\xi \overline{\omega}).
\end{equation}

\end{example}
\begin{example}
Let $E(n)$ be the Euclidean group which is the semi-direct product of $So(n) \times_\tau \mathbb{R}^n$ where  the continuous homomorphism $\tau:So(n) \rightarrow Aut(\mathbb{R}^n)$ given by $\sigma \mapsto \tau_\sigma$ via $\tau_\sigma(x)=\sigma x$, for all $x \in{\mathbb{R}^n}$. The group operation for $E(n)$ is 
$$(\sigma_1,x_1)\times_\tau (\sigma_2,x_2)=(\sigma_1\sigma_2,x_1+\sigma_1x_2).$$ Consider the  continuous homomorphism $\hat{\tau}: So(n) \rightarrow Aut(\mathbb{R}^n)$ via $\sigma \mapsto \hat{\tau_\sigma}$ which is given by $\hat{\tau_\sigma}((\omega)=\omega_\sigma=\omega \circ \tau_{\sigma^{-1}}$. Thus the generalized Weyl Heisenberg group  of $E(n)$, is the set $\mathbb{H}(E(n))=(So(n)\times_\tau \mathbb{R}^n) \times _\theta (\mathbb{R}^n \times  \mathbb{T})$ with the group  operation 
\begin{equation*}
 (\sigma_1,x_1,\omega_1,z_1)(\sigma_2,x_2,\omega_2,z_2)=(\sigma_{\sigma_2},x_1+\sigma_1 x_2,\omega_1(\omega_2)_{\sigma_1},(\omega_2)_{\sigma_1}(x_1)z_1z_2),
\end{equation*}
for all $(\sigma_1,x_1,\omega_1,z_1)(\sigma_2,x_2,\omega_2,z_2) \in{\mathbb{H}(E(n))}$ and with the product topology. Then the  square integrable representation $\pi$ of $\mathbb{H}(E(n))$ onto $L^2(\mathbb{R}^n)$ is 
$$\pi(\sigma,x,\omega,z)f(\xi)=e^{2\pi ix(\xi-\omega)}f((\xi-\omega)_{\sigma^{-1}}).$$ Note that $H$ is compact  and $\delta_H(h)=1$.
\end{example}
\begin{example}
Let $\mathbb{H}(\mathbb{R}^n)=\mathbb{R}^n \times_\theta(\mathbb{R}^n \times \mathbb{T})$ be the classical  Heisenberg group on $\mathbb{R}^n,$ in which the continuous homomorphism $x \mapsto \theta_x$ from $\mathbb{R}^n$ into $Aut(\mathbb{R}^n \times \mathbb{T}) $
 is defined by $\theta_x(y,z)=(y,z e^{2\pi ix.y}).$ Then the  square integrable representation $\pi$ 
 of $\mathbb{H}(\mathbb{R}^n)$ onto $L^2(\mathbb{R}^n)$ is $$ \pi(x,\omega,z)f(\xi)=z.e^{2\pi i x(\xi-\omega)}f(\xi-\omega).$$

\end{example}
 

\bibliographystyle{amsplain}

\end{document}